\documentclass[a4paper,withmarginpar]{article}

\usepackage[utf8x]{inputenc}
\usepackage[T1]{fontenc}
\usepackage[a4paper,top=3cm,bottom=2cm,left=3cm,right=3cm,marginparwidth=1.75cm]{geometry}

\usepackage{enumitem}
\usepackage{verbatim}
\input{xy}
\xyoption{all}
\usepackage[all]{xy} 
\usepackage{amsmath}
\usepackage{graphicx}
\usepackage{mathtools}
\usepackage[colorinlistoftodos]{todonotes}
\usepackage[colorlinks=true, allcolors=blue]{hyperref}
\usepackage{amssymb}
\usepackage{amsthm}
\theoremstyle{plain}
\theoremstyle{plain}
\newtheorem{thm}{\textsf{\textbf{Theorem}}}[section]

\newtheorem*{thm*}{\textsf{\textbf{Theorem}}}

\theoremstyle{definition}
\newtheorem{dfn}[thm]{\textbf{\textsf{Definition}}}

\newtheorem{rem}[thm]{{\textsf{Remark}}}

\DeclareMathOperator{\clu}{cl}

\DeclareMathOperator{\mesh}{mesh}
\DeclareMathOperator{\vx}{V}
\DeclareMathOperator{\edges}{E}

\DeclareMathOperator{\leb}{l}
\DeclareMathOperator{\diam}{diam}
\DeclareMathOperator{\topol}{top}

\DeclareMathOperator{\dist}{dist}

\newcommand{\Cee }{\mathcal C}
\newcommand{\Gee }{\mathcal G}

\newcommand{\Fee }{\mathcal F}
\newcommand{\nat }{\mathbb N}
\newcommand{\Uee }{\mathcal U}
\newcommand{\Vee }{\mathcal V}

\newcommand{\See }{\mathcal S}
\title{Graph covers of higher dimensional dynamical systems}
\author{Przemysław Kucharski\footnote{
		 }}

\newcommand{\cl}[1]{\clu{#1}}
\newcommand{\p}[1]{\hat{#1}}
\newcommand{\eps}[1]{(#1)^{\epsilon_{i-1}}}
\newcommand{\epsi}[1]{(#1)^{\epsilon_{i}}}

\newcommand{\indexset}[1]{\iota(#1)}
\providecommand{\keywords}[1]
{
	\small	
	\textbf{\textit{Keywords---}} #1
}
\providecommand{\declarationofinterest}[1]
{
	\small	
	\textbf{\textit{Declaration of interest---}} #1
}
\begin{document}
	\maketitle
\begin{abstract}
	We generalize the notion of an inverse sequence of graph covers from the zero-dimensional dynamical systems \cite{Shimomura2014} to any dynamical system.
\end{abstract}
\\
\declarationofinterest{Author declares that there are no financial or personal relationships with other people or organizations that could inappropriately influence this work.}
\section*{Introduction}
In the present paper, all spaces under consideration are compact and metric. In 2006 Gambaudo and Martens proved that every minimal dynamical system on a Cantor set can be obtained as an inverse limit of directed graphs. This result allowed them to connect minimal dynamics on a Cantor set with its homology, homotopy and cohomology. In 2012 Bernardes and Darji \cite{BD} used similar approach to study the comeager conjugacy class of the homeomorphism group of the Cantor set. In 2014 Shimomura generalized the methods of Gambaudo and Martens to all zero-dimensional systems \cite{Shimomura2014}, resulting in Theorem \ref*{shimomura-zerdim-charact}. In \cite{CHL} Clark, Hurder and Lukina generalized the approach to obtain representations of suspensions of pseudogroup actions on Cantor sets as inverse limits of branched manifolds. Since then graph covers have proved to be a powerful tool in dynamics. Among many results, they have been used to  construct completely scrambled Cantor systems, with various additional properties; see \cite{Shimomura_2015} and \cite{boronski_kupka_oprocha_2019}. In  2016 Shimomura characterised one-sided Toeplitz flows using graph covers \cite{SHIMOMURA201663}, and then studied the relation between the ergodic measures of a system and circuits of graph covers \cite{Shimomura2016ergodicity}, expanding on conditions of graph covers for minimal systems. In their recent breaktrough, Good and Meddaugh \cite{goodmedd} used similar methods to determine which compact zero-dimensional dynamical systems have shadowing property. Also very recently, a relation between Bratelli-Vershik diagrams and graph covers has been established \cite{Shimomura_2020}. In \cite{cantorslow} the authors used the method to show that any Cantor minimal system can be embedded into the real line with vanishing derivative everywhere, see also \cite{Boroski2017EdreisCR}.
It seems that it is only a matter of time before the method\footnote{The method is often referred to as "dynamical coding".} will find new applications in the study of other dynamical invariants, such as topological entropy. 

Therefore it seems very desirable to know if it is possible to push the graph covers method outside the zero-dimensional setting. Analysing proofs of Gambaudo and Martens from \cite{Gambaudo2006AlgebraicTF}, and then proofs of Shimomura from \cite{Shimomura2014} and \cite[Thm. 3.4.]{zerodim-takashi-BV} we realize that for an arbitrary dynamical system $(X,f)$, without a clopen base, we cannot uniquely determine points in $X$. Therefore the construction similar to \cite[Thm. 3.4.]{zerodim-takashi-BV} will lead us at best to a system semi-conjugate to $(X,f)$, which is not good enough for most applications. In this article we introduce a new kind of graph covers, called an inverse sequence of twinned graph homomorphisms, that remedies the above problem, and allows us to construct by the way of graph covers a conjugacy of a given system.
 
\section{Preliminaries}
For a subset $U$ of a topological space $X$ we denote by $\cl U$ the closure of $U$. A pair $(X,f)$, where $X$ is a compact metric space and $f \colon X \rightarrow X$ is a continuous mapping will be called a \emph{dynamical system}. Note that we do not demand that $f$ is surjective. If $d$ is a metric on $X$, $q\in X$ a point and $V\subset X$ any subset, we define the distance between $q$ and $V$ by $\dist(q,V):=\inf_{y\in V}d(q,y)$. For $\epsilon>0$, we define an $\epsilon$-neighbourhood of $V$ by $B(V,\epsilon):=\{x\in X\colon \dist(x,V)<\epsilon\}$. We will say that a dynamical system $(Y,g)$ is a \emph{factor} of $(X,f)$, if there exists a continuous surjective mapping $\phi \colon X \rightarrow Y$, called a factor map, such that $g \circ \phi =\phi \circ f$. On the other hand if $\phi$ is a homeomorphism it will be referred to as an \emph{isomorphism} and we will say that systems $(Y,g)$ and $(X,f)$ are conjugated. Diameter of a subset $U$ of a space $X$ will be denoted by $\diam U$. Let $\mathcal{U}$ be a family of subsets of $X$ we define the mesh number of $\mathcal{U}$ as $\mesh \mathcal{U}= \sup_{U \in \mathcal{U}} \diam(U)$. We will write $\leb (\mathcal{U})$ for the \emph{Lebesgue number} of an open cover $\mathcal{U}$. Any function $f\colon X\to X$ defines a new family $f(\Uee):=\{f(U)\colon U\in\Uee\}$. To simplify notation, by writing $x \in \prod_{i=0}^{\infty}K_i$, we will always mean $x=(x_i)^{\infty}_{i=0}$. We will sometimes omit scripts and write simply $(x_{i})$.  

Let $V$ be a finite set. A pair $G=(V,E)$ will be called a (non directed) \emph{graph}, if the set of \emph{edges} $E$ is a binary (symmetric) relation on the set $V$ of \emph{vertices}. We will sometimes write $\vx(G)=V$ and $\edges(G)=E$. If for every vertex $v\in V$ we can find $u,w\in V$ with $(u,v),(v,w)\in E$, then we will say that $G$ is edge surjective. Given (non directed) graphs $G_1=(V_1,E_1)$ and $G_2=(V_2,E_2)$ a mapping $\phi \colon V_1 \rightarrow V_2$ will be called \emph{a (non directed) graph homomorphism}, if for any pair $(u,v) \in E_1$ we have $(\phi (u),\phi (v)) \in E_2$. Function $\phi$ naturally induces mapping defined on the set of edges, we will denote it by the same letter, and write $\phi \colon G_1 \rightarrow G_2$. Moreover if $\phi(E_{1})=E_{2}$, then we will refer to $\phi$ as an edge-surjective homomorphism. After Shimomura \cite{Shimomura2014}, a graph homeomorphism $\phi$ will be called $+$directional if for any edges $(u,v),(u,w)\in E_{1}$ we have $\phi(v)=\phi(w)$. If $G_{1}$ and $G_{2}$ are edge surjective and $\phi$ is $+$directional edge-surjective homomorphism then $\phi$ will be called a graph cover.

Let $\mathcal{R}\subset V\times W$ and $\mathcal{Q}\subset W\times U$ be binary relations. By composition $\mathcal{Q}\mathcal{R}$ we mean a relation on $V\times U$ defined by $(v,u)\in \mathcal{Q}\mathcal{R}$ iff there exists $w\in W$ such that $(v,w)\in \mathcal{R}$ and $(w,u)\in \mathcal{Q}$. We also denote by $\mathcal{R}v=\{w\in W~|~(v,w)\in\mathcal{R}\}$ the image of $v$ under $\mathcal{R}$. If $\mathcal R=V\times V$ is an equivalence relation, then by $[e]_{\mathcal R}\in V/_{\mathcal R}$ we define the equivalence class of $e\in V$. We define a projection on equivalence classes $\pi_{\mathcal R}(v)=[v]_{\mathcal R}$. For better clarity, we will sometimes omit brackets and write $\pi_{\mathcal R}v$.

On the product $X=\prod_{i=0}^{\infty}X_i$ of topological spaces $X_{i}$ we will always consider a Tychonoff product topology. Assume that $X_m$ are finite, for every $m\in\nat$. Then for $n \in \mathbb{N}$, a \emph{cylinder} $\Cee(w,n)$ defined by a sequence $w \in \prod_{i=0}^{n}X_i$ is given by the formula $\Cee(w,n) = \{ y \in X:  w = (y_0,...,y_{n}) \}$. Note that on the inverse limit of sequence $(X_i,\rho_{i})$ it is enough to specify the last element of sequence $w$, that is cylinders are given by $\Cee(x_{n},n)=\{ y \in \varprojlim (X_i,\rho_{i}):  y_{n} = x_{n} \}$ for $x_{n}\in X_{n}$. Similarly we define $\Cee(x,n)=\Cee(x_{n},n)$ for $x=(x_{i})\in \varprojlim (X_i,\rho_{i})$.

\section{Inverse limits of graph covers}
Following Shimomura \cite{Shimomura2014} we define an inverse limit of graph covers.
\begin{dfn}
	An inverse sequence $\mathcal{G}=(G_i,\phi_{i})_{i\in \nat}$ will be called an inverse sequence of graph homomorphisms if for every $i\in\nat$ mapping $\phi_{i}\colon G_{i}\to G_{i-1}$ is a graph homomorphism. If additionally, $\phi_{i}$ is a graph cover, $\mathcal{G}$ will be called an inverse sequence of graph covers.
\end{dfn}
\begin{dfn}\label{inv-limit-graphs}
	Let $(G_i,\phi_{i})$ be an inverse sequence and $\phi_{i}\colon G_{i}\to G_{i-1}$ be graph homomorphisms. We define the inverse limit of graph homomorphisms $(V_{\mathcal{G}},E_{\mathcal{G}})$ by
	\[
	V_{\mathcal{G}}:=\{(x_0,x_1,...)\in \prod_{i=0}^{\infty}V_i: x_i=\phi_{i+1}(x_{i+1}),i=0,1,...\} = \varprojlim (V_i,\phi_{i})
	,\]
	\[
	E_{\mathcal{G}}:=\{(x,y)\in V_{\mathcal{G}}\times V_{\mathcal{G}}: (x_i,y_i)\in E_i, i=0,1,...\}
	.\]
\end{dfn}
\begin{rem}
	Note that $E_{\mathcal{G}}$ is a binary relation on non empty compact metric space $V_{\mathcal{G}}$. In general $E_{\mathcal{G}}$ does not need to be a function, but if $(G_i,\phi_{i})$ is additionally an inverse sequence of graph covers, that is functions $\phi_{i}$ are edge-surjective, positive directional homomorphisms and $G_{i}$ are edge surjective, it is a matter of simple check, that $E_{\mathcal{G}}$ is a continuous function, as stated in Theorem \ref{c-grafow} (compare with \cite[Lem. 3.5]{Shimomura2014}).
\end{rem}
\begin{thm} \label{c-grafow}
	Let $\mathcal{G}=(G_i,\phi_i)$ be an inverse sequence of graph covers. Then a pair $(V_{\mathcal{G}},E_{\mathcal{G}})$ is zero-dimensional surjective dynamical system.
\end{thm}
Below we present slightly extended version of \cite[Lem. 3.5]{Shimomura2014}.
\begin{dfn}\label{df1}
	For a function $f \colon X \rightarrow X$ and arbitrary cover $\mathcal{U}$ of a topological space $X$, we define the following relation
	\[
	f^{\mathcal{U}}:=\{(U,V) \in \mathcal{U} \times \mathcal{U} ~ | ~ f(U)\cap V \neq \emptyset\}.
	\]
\end{dfn}
\begin{thm}\label{zero->ciag-odw}
	Let $(X,f)$ be a surjective zero-dimensional dynamical system. Then there exists a sequence $ \{\mathcal{U}_{i}\}_{i\in\nat} $ of partitions of $X$ into clopen covers with $\lim_{i\to\infty}\mesh\mathcal{U}_{i} =0$ and a sequence of mappings $\{\phi_{i}\colon\mathcal{U}_{i}\to\mathcal{U}_{i-1}\}_{i\in\nat}$ induced by inclusion, such that $\mathcal{G}=((\mathcal{U}_{i},f^{\mathcal{U}_{i}}),\phi_i)$ is an inverse sequence of graph covers and $(V_{\mathcal{G}},E_{\mathcal{G}})$ is conjugate to $(X,f)$.
\end{thm}

Combining Theorems \ref{c-grafow} and \ref{zero->ciag-odw} we obtain 
\begin{thm}\label{shimomura-zerdim-charact}
	A surjective dynamical system $(X,f)$ is zero-dimensional if and only if it is conjugate to an inverse limit of graph covers.
\end{thm}

\section{Inverse sequences of twinned graph homomorphisms}

\begin{dfn}
	We will call a pair $(\mathcal{G},\mathcal{F})$ an inverse sequence of twinned graph homomorphisms, if
	\begin{enumerate}[label=(DS\arabic*)]
		\item\label{DS0} $|\vx(G_0)|=1$,
		\item\label{DS1} $\mathcal{G}=(G_i,\phi_i)$ is an inverse sequence of graphs homomorphisms, every vertex in $G_{i}$ has outgoing edge and $\phi_{i}$ are edge-surjective for $i>0$,
		\item\label{DS2} $\mathcal{F}=(F_i,\phi_i)$ is an inverse sequence of non directed graph homomorphisms such that $\vx(G_{i})=\vx(F_{i})$, $(v,v)\in \edges(F_{i})$ for every $v\in\vx(F_{i})$ and $i\geq0$,
		\item\label{DS3} if $(a,b)\in \edges (F_i)$ and $(a,a'),~(b,b')\in E(G_{i})$ for some $a',~b'\in V(G_{i})$, then $(\phi_{i}(a'),\phi_{i}(b'))\in E(F_{i-1})$, where $i>0$,
		\item\label{DS3b} if $(a,b),(b,c)\in \edges (F_i)$, then $(\phi_{i}(a),\phi_{i}(c))\in\edges (F_{i-1})$, where $i>0$.
	\end{enumerate}
\end{dfn}
\begin{rem}
	Note that Theorem \ref{shimomura-zerdim-charact} requires dynamical system to be surjective. If we want to obtain any zero-dimensional dynamical system as an inverse limit of graph covers, we would have to adopt different definition of graph cover, one that does not require graphs to be edge-surjective. For greater generality and clarity we do not impose assumption of surjectivity in our characterisation of higher dimensional dynamical systems.
\end{rem}
Following Definition \ref{inv-limit-graphs} we can define relations $E_{\Gee}$, $E_{\Fee}$ and set $V_{\Gee}$. Then quotient space ${V_{\mathcal{G}}}/_{E_{\mathcal{F}}}$ is well defined and gives rise to a new dynamical system. Let $\pi\colon V_{\mathcal{G}}\to {V_{\mathcal{G}}}/_{E_{\mathcal{F}}}$ be the projection on equivalence classes.
\begin{thm}\label{graph->space, nonzero}
	If $(\mathcal{G},\mathcal{F})$ is an inverse sequence of twinned graph homomorphisms, then $E_{\mathcal{F}}$ is an equivalence relation and $(Y,T)$ is a dynamical system, where $Y={V_{\mathcal{G}}}/_{E_{\mathcal{F}}}$ and $T$ is defined by the formula $\{T(\pi v)\}=\pi E_{\mathcal{G}}(v)$ for $v\in V_{\Gee}$.
\end{thm}
\begin{proof}
	We prove that $E_{\Fee}$ is an equivalence relation. Reflexivity and symmetry of $E_{\mathcal{F}}$ are direct consequences of graphs $F_{i}$ being non directed. Transitivity is implied by \ref{DS3b}.
	
	We consider $Y$ as a topological space with the quotient topology. Firstly, we will prove that $T$ is well defined. Let $x,x'\in V_{\mathcal{G}}$ satisfy $\pi x=\pi x'$. Take any $y,y'\in V_{\Gee}$ with $y\in E_{\Gee}x$ and $y'\in E_{\Gee}x'$. As every vertex has outgoing edge and $\phi_{i}$ are edge-surjective, such points always exist. Fix $i\in\nat$, then $(x_{i},y_{i}),(x'_{i},y'_{i})\in E(G_{i})$. Using property \ref{DS3} we obtain $(y_{i-1},y'_{i-1})=(\phi_{i}(y_{i}),\phi_{i}(y'_{i}))\in E(F_{i-1})$. As $i$ was arbitrary, we learn that $\pi y=\pi y'$. Consequently, $\pi E_{\mathcal{G}}(x)=\pi E_{\mathcal{G}}(x')$ is a singleton and $T(\pi x)=T(\pi x')$, which proves that $T$ is well-defined.
	
	We check that $T$ is continuous. Let $\overline{\Cee}(x,j)=\bigcup\{\Cee(v,j)~|~(v,x_{j})\in E(F_{j})\}$ and extend definition further, for any $A\subset V_{\Gee}$, put $\overline{\Cee}(A,j)=\bigcup_{a\in A}\overline{\Cee}(a,j)$. Inductively define, for $i>j$, $\overline{\Cee}(x,j,i)=\overline{\Cee}(\overline{\Cee}(x,j,i-1),i)$ and $\overline{\Cee}(x,j,j)=\overline{\Cee}(x,j)$. Finally, put $\tilde{\Cee}(x,j)=\bigcup_{i\geq j}\overline{\Cee}(x,j,i)$ and $\tilde{\Cee}(A,j)=\bigcup_{x\in A}\tilde{\Cee}(x,j)$. Let us note few facts about those sets. We first show that for every  $x'\in\tilde{\Cee}(x,j)$  we have $\pi x'\subset \tilde{\Cee}(x,j)$. Let $y\in\pi x'$, then $(x'_{n},y_{n})\in E(F_{n})$ for every $n\in\nat$. One can find $i_{0}\geq j$ with $x'\in\overline{\Cee}(x,j,i_{0})$. Note that $\overline{\Cee}(x,j,i_{0})$ is a sum of cylinders, so we choose $i_{1}\geq i_{0}$ with $\Cee(x',i)\subset \overline{\Cee}(x,j,i_{0})$ for every $i>i_{1}$. We observe \[ \overline\Cee(\Cee(x',i),i+1)\subset \overline\Cee(\overline{\Cee}(x,j,i_{0}),i+1)\subset \overline\Cee(\overline{\Cee}(x,j,i_{0}),i_{0}+1)\subset \overline{\Cee}(x,j,i_{0}+1)\subset \tilde{\Cee}(x,j).\] As $y\in\Cee(y,i+1) \subset \overline\Cee(\Cee(x',i),i+1)$, we have proved $y\in \tilde{\Cee}(x,j)$.
	
	As a consequence of above, $\pi^{-1}\pi(\tilde{\Cee}(x,j))=\tilde{\Cee}(x,j)$, which together with openness of $\tilde{\Cee}(x,j)$ implies $\pi(\tilde{\Cee}(x,j))$ is open in $Y$. Note that if $\pi y,\pi y'\in\pi(\tilde{\Cee}(x,j))$, then $(y_{n},y'_{n})\in E(F_{n})$ for $n$ up to $j$. We conclude that $\bigcap _{j\geq 0}\pi(\tilde{\Cee}(x,j))=\{\pi x\}$. The mapping $\pi$ is continuous by definition, hence $X$ is compact. Moreover, $\tilde{\Cee}(x,i'')\subset\tilde{\Cee}(x,i')$ for $i''>i'$, and so sets of the form $\pi(\tilde{\Cee}(x,j))$ generate topology. Note that if $y\in E_{\Gee}x$, then $\pi(\tilde{\Cee}(y,j))$ contains $\pi y =T(\pi x )$, thus $\pi(\tilde{\Cee}(y,j))$, as well as $\pi (\tilde{\Cee}(E_{\Gee}x,j ))$, is an open neighbourhood of $T(\pi x)$. Therefore, to check continuity of $T$ for every $x\in X$ and $k\geq 0$ we have to find $j\geq 0$ such that \[ T\pi(\tilde{\Cee}(x,j))=\pi E_{\Gee}(\tilde{\Cee}(x,j))\subset \pi(\tilde{\Cee}(E_{\Gee}x,k)).\] It is enough to show $E_{\Gee}(\tilde{\Cee}(x,j))\subset \tilde{\Cee}(E_{\Gee}x,k)$. Put $j=k+1$. By induction on $i\geq k+1$, we will show $E_{\Gee}(\overline{\Cee}(x,k+1,i))\subset \overline{\Cee}(E_{\Gee}x,k,i-1)$. Let $i=k+1$ and $x'\in \overline{\Cee}(x,k+1)$. Then $(x'_{k+1},x_{k+1})\in E(F_{k+1})$ and by property \ref{DS3}, $(y'_{k},y_{k})\in E(F_{k})$ for every $y'\in E_{\Gee}x'$ and $y\in E_{\Gee}x$, showing $E_{\Gee}(\overline{\Cee}(x,k+1))\subset \overline{\Cee}(E_{\Gee}x,k)$. Let us assume $E_{\Gee}(\overline{\Cee}(x,k+1,i))\subset \overline{\Cee}(E_{\Gee}x,k,i-1)$ for some $i\geq k+1$. Take $x'\in \overline{\Cee}(x,k+1,i+1)$. It satisfies $(x'_{i+1},x''_{i+1})\in E(F_{i+1})$ for some $x''\in \overline{\Cee}(x,k+1,i)$. Once again, by property \ref{DS3}, for every $y'\in E_{\Gee}x'$ and $y''\in E_{\Gee}x''$ we have $(y'_{i},y''_{i})\in E(F_{i})$. Consequently, \[ y'\in \overline{\Cee}(E_{\Gee}(\overline{\Cee}(x,k+1,i)),i)\subset \overline{\Cee}(\overline{\Cee}(E_{\Gee}x,k,i-1),i)=\overline{\Cee}(E_{\Gee}x,k,i),\]which proves $E_{\Gee}(\overline{\Cee}(x,k+1,i+1))\subset \overline{\Cee}(E_{\Gee}x,k,i)$ and concludes proof of continuity of $T$.
\end{proof}
For the sake of the following proof, let us define a disjoint union $\coprod_{i\in I}\See_{i}$ of families $\{\See_{i}\}_{i\in I}$ indexed by a set $I$ by $\coprod_{i\in I}\See_{i}=\{(S,i)\colon S\in \See_{i}\}$ and projections \[ \coprod_{i\in I}\See_{i}\ni s=(S,i)\mapsto \p s =S\in\bigcup_{i\in I}\See_{i}\text{ and } \]
\[ \coprod_{i\in I}\See_{i}\ni s=(S,i)\mapsto \indexset s =i\in I. \]
Recall $\pi\colon V_{\mathcal{G}}\to {V_{\mathcal{G}}}/_{E_{\mathcal{F}}}$ is the projection on equivalence classes.
\begin{thm} \label{space->graph,nonzero}
Let $(X,f)$ be a dynamical system. Then there exists an inverse sequence $(\mathcal{G},\mathcal{F})$ of twinned graph homomorphisms such that $(X,f)$ is conjugate to $(Y,T)$, where $Y={V_{\mathcal{G}}}/_{E_{\mathcal{F}}}$ and $T$ is defined by $\{T(\pi v)\}=\pi E_{\mathcal{G}}(v)$ for $v\in V_{\mathcal{G}}$.
\end{thm}
\begin{proof}
	For any open set $U$ and $\epsilon>0$ put $U^{\epsilon}=B(U,\epsilon)$. Similarly, for any open cover $\mathcal U$, let $\mathcal U^{\epsilon}=\{U^{\epsilon}\}_{U\in\mathcal U}$. Let $\{\Uee_{i}\}_{i\in\nat}$ be a sequence of finite, open, covers of $X$ satisfying the following conditions for every $i=1,2,\ldots$
	\begin{enumerate}[label=(C\arabic*)]\label{conditions on covers}
		\item \label{c1}$\mathcal{U}_0=\{X\}$ and $\emptyset\notin \Uee_{i}$,
		\item \label{c2}cover $\mathcal U_{i}^{\epsilon_{i}}$ refines $\mathcal{U}_{i-1}$, where $\epsilon_{i}=\max\{\mesh f(\mathcal{U}_{i}),\mesh \mathcal{U}_{i}\}$, and \[ 2\max\{\mesh f(\mathcal U_{i}^{\epsilon_{i}}),\mesh \mathcal U_{i}^{\epsilon_{i}}\}<\epsilon_{i-1} \]
		\item \label{c4}we have $\mesh \mathcal{U}_{i} \leq 2^{-i}$, 
		\item\label{c5} we have $V=\bigcup \Uee_{i,V}$ for every $V\in \mathcal U_{i-1}$, where $\Uee_{i,V}=\{U\in \mathcal U_{i}\colon U\subset V\}$.
	\end{enumerate}
	Sequence $\{\mathcal{U}_i\}_{i=0}^{\infty}$ can be defined inductively. Let us suppose $\mathcal{U}_{i-1}$ is constructed, $i>0$. There exists a sequence of open covers $\Vee_{n}$, $n\in\nat$, such that $\Vee_{n}$ refines $\Vee_{n-1}$, $\mesh \Vee_{n}\to0$ as $n\to\infty$, and every element of $\Uee_{i-1}$ is a union of some elements of $\Vee_{n}$. Put $\epsilon'_{n}=\max\{\mesh f(\mathcal{V}_{n}),\mesh \mathcal{V}_{n}\}$. Then $\alpha_{n}=\max\{\mesh \mathcal V_{n}^{\epsilon'_{n}},\mesh f(\mathcal V_{n}^{\epsilon'_{n}})\}\to 0$ as $n\to\infty$. Let $n_{0}\in\nat$ be big enough so that $\alpha_{n_{0}}<\min\{\epsilon_{i-1},2^{-i},\leb(\Uee_{i-1})\}$. It follows that conditions \ref{c2} to \ref{c5} are satisfied with $\Uee_{i}:=\Vee_{n_{0}}$.
	
	Define $G_{0}$ and $F_{0}$ by \[ V(G_{0})=\coprod\Uee_{0}=\{(X,0)\},\quad E(G_{0})=\{((X,0),(X,0))\}\]
	\[ V(F_{0})=V(G_{0}),\quad E(F_{0})=E(G_{0}).\] Then inductively define two inverse sequences
	\[
	\mathcal{G} :=G_{0} \xleftarrow{\phi_1}  G_{1} \xleftarrow{\phi_2}...,
	\]
	\[
	\mathcal{F} :=F_{0} \xleftarrow{\phi_1}  F_{1} \xleftarrow{\phi_2}...,
	\]
	where \[ V(G_{i})=\coprod_{v\in V(G_{i-1})}\Uee_{i,\p v},\quad (a,b)\in E(G_{i})\text{ iff } f(\p a)\cap \p b\neq\emptyset\text{ and }\]
	\[ V(F_{i})=V(G_{i}),\quad (a,b)\in E(F_{i})\text{ iff } (\p a)^{\epsilon_{i}}\cap (\p b)^{\epsilon_{i}}\neq\emptyset. \]
	Finally, we put $\phi_{i}(a)=\iota(a)$. Note that $\phi_{i}(a)=c$ is an element of disjoint union, hence we can apply projection to $c$ to obtain $\p c\in\mathcal U_{i-1}$. It will be denoted by $\p c =\p \phi_{i}(a)$. We check that $(\mathcal G,\mathcal F)$ is an inverse sequence of twinned graph homomorphisms. For every $i>0$ mapping $\phi_{i}$ is a homomorphism for both $\Gee$ and $\Fee$. Indeed, if $f(\p a)\cap \p b\neq\emptyset$, then, as $\p a\subset \p \phi_{i}(a)$ and $\p b\subset \p \phi_{i}(b)$, we deduce $f(\p a)\cap \p b\subset f(\p\phi_{i}(a))\cap \p\phi_{i}(b)$ is non empty. Analogously, if $(\p a)^{\epsilon_{i}}\cap (\p b)^{\epsilon_{i}}\neq\emptyset$, then $(\p a)^{\epsilon_{i}}\cap (\p b)^{\epsilon_{i}}\subset(\p \phi_{i}(a))^{\epsilon_{i-1}}\cap (\p \phi_{i}(b))^{\epsilon_{i-1}}$ is non empty. Moreover, $\phi_{i}\colon G_{i}\to G_{i-1}$ are edge surjective. Certainly, if $(a,b)\in E(G_{i-1})$, then $f(\p a)\cap \p b\ne\emptyset$ and so, as $\p a=\bigcup \mathcal U_{i,\p a}$ and $\p b=\bigcup \mathcal U_{i,\p b}$, we can find $a',b'\in V(G_{i})$ with $f(\p a')\cap \p b'\neq\emptyset$ and $\phi_{i}(a')=a$, $\phi_{i}(b')=b$. It implies $(a',b')\in E(G_{i})$. We check \ref{DS3}. Let $i>0$. If $(a,b)\in E(F_{i})$ and $(a,a'),(b,b')\in E(G_{i}) $, then $\epsi{\p a}\cap \epsi{\p b}\neq\emptyset$, $f(\p a)\cap \p a'\neq\emptyset$ and $f(\p b)\cap \p b'\neq\emptyset$. Consequently, $f(\epsi{\p a})\cap f(\epsi{\p b})\neq\emptyset$.  As \[ \max\{\diam(\p a'\cup f(\epsi{\p a})),\diam(\p b'\cup f(\epsi{\p b}))\}<2\max\{\mesh\mathcal U_{i},\mesh f(\mathcal U_{i}^{\epsilon_{i}})\}<\epsilon_{i-1}, \]  we have \[ f(\epsi{\p a})\subset \eps{\p \phi_{i}(a')}\text{ and }\] \[ f(\epsi{\p b})\subset \eps{\p \phi_{i}(b')}. \] We conclude \[ \emptyset\neq f(\epsi{\p a})\cap f(\epsi{\p b})\subset \eps{\p \phi_{i}(a')}\cap \eps{\p \phi_{i}(b')}. \] Therefore $(\phi_{i}(a') , \phi_{i}(b'))\in E(F_{i-1})$. Condition \ref{DS3b} is proved similarly. If $(a,b),(b,c)\in E(F_{i})$, then $\epsi{\p a}\cap\epsi{\p b}\neq\emptyset$ and $\epsi{\p b}\cap\epsi{\p c}\neq\emptyset$. It follows that \[ \max\{\diam(\epsi{\p a}\cup \epsi{\p b}),\diam(\epsi{\p b}\cup \epsi{\p c})\}<2\mesh \mathcal U_{i}^{\epsilon_{i}}<\epsilon_{i-1}, \]and consequently \[ 
	\epsi{\p b}\subset \eps{\p a}\subset\eps{\p \phi_{i}(a)}\text{, and }  \]\[  
	\epsi{\p b}\subset \eps{\p c}\subset\eps{\p \phi_{i}(c)}.  \]Hence $\eps{\p \phi_{i}(a)}\cap \eps{\p \phi_{i}(c)}$ is non empty and $(\phi_{i}(a),\phi_{i}(c))\in E(F_{i-1})$.
	
	Therefore, the pair $(\Gee,\Fee)$ is an inverse sequence of twinned graph homomorphisms, which gives rise to a dynamical system $(Y,T)$.
	
	We will now show that dynamical systems $(X,f)$ and $(Y,T)$ are conjugate. Noting that for every $x\in V_{\mathcal G}$ intersection $\bigcap_{i=0}^{\infty} \epsi{\p x_i}=\bigcap_{i=0}^{\infty} \clu \p x_i$ is a singleton by \ref{c2},\ref{c4} and \ref{c5}, we define the function 
	\[ \psi \colon Y\ni \pi x\mapsto  \psi(\pi x)\in X \] to satisfy
	\[
	\{\psi(\pi x)\} = \bigcap_{i=0}^{\infty} \clu\p x_i \subset X.
	\]
	We will prove the following.
	\begin{description} 
		\item[Function $\psi$ is well defined.] Let $x,y \in V_{\mathcal{G}}$ be such that $(x,y)\in E_{\Fee}$. Note that $\{\clu \p x_{i}\}$, $\{\clu \p y_{i}\}$ and $\{ \clu \epsi{\p x_{i}}\cap \clu \epsi{\p y_{i}}\}$ are decreasing sequences of non empty compact sets, with diameters tending to zero. We conclude $\bigcap_{i=0}^{\infty}\clu \p x_i=\bigcap_{i=0}^{\infty}\clu\epsi{ \p  x_i}\cap\clu\epsi{ y_i}=\bigcap_{i=0}^{\infty}\clu \p y_i$.
		\item[Function $\psi$ is surjective.] It is a direct consequence of the construction.
		\item[Function $\psi$ is injective.] Let $x,y \in V_{\mathcal{G}}$ and suppose that $\psi(\pi x)=\psi(\pi y)$. Then $\psi(\pi x) \in \clu \p x_i \cap \clu \p y_i $, for every $i \in \mathbb{N}$, hence $\epsi{\p x_i} \cap \epsi{\p y_i}\neq\emptyset$ and in consequence $\pi x=\pi y$.
		\item[Function $\psi$ is continuous.] Let $U \in \topol(X)$. Showing $\psi^{-1}(U)\in \topol(Y)$, amounts to proving that $Z=\bigcup\psi^{-1}(U)=\pi^{-1}\psi^{-1}(U)$ is open. Fix any $a \in Z$, we then have $\{\psi(\pi a)\}=\bigcap_{i=0}^{\infty}\clu \p a_i \subset U$. As $X$ is compact and $U$ open, we deduce that there exists $l \in \mathbb{N}$ such that $\bigcap_{i=0}^{l}\clu \p a_i \subset U$. Let $b\in\Cee(a_{l},l)$. It follows $\psi(\pi(b)) \in \bigcap_{i=0}^{l}\clu \p  a_i \subset U$ and so $\Cee(a_{l},l)\subset \pi^{-1}\psi^{-1}(U)$.
		\item[Function $\psi$ commutes with $T$ and $f$.] Let $x,y \in V_{\mathcal{G}}$ be such that $T(\pi x)=\pi y$. Then 
			\[
			\{\psi(T(\pi x))\}=\bigcap_{i=0}^{\infty}\clu \p y_i
			\] and
			\[
			\{f(\psi(\pi x))\}=f(\bigcap_{i=0}^{\infty}\clu \p x_i)
			.\] 
			We can find $a_{x}$ and $a_{y}$ in $X$ such that
			\begin{gather*}
			\bigcap_{i=0}^{\infty}\clu \p x_i=\{a_{x}\}\text{ and }
			\bigcap_{i=0}^{\infty}(f(\clu \p x_{i})\cap \clu \p y_i)=\{a_{y}\},
			\end{gather*}
			which implies that $f(a_{x})=a_{y}$. Finally, we have
			\[
			\{f(\psi(\pi x))\}=f(\bigcap_{i=0}^{\infty}\clu \p x_i)=f(a_{x})=a_{y}=\bigcap_{i=0}^{\infty}\clu \p y_i=\{\psi(T(\pi x))\} .\qedhere
			\]
	\end{description}
\end{proof}

Taking into account Theorems \ref{graph->space, nonzero} and \ref{space->graph,nonzero} we obtain
\begin{thm}\label{mainres}
	Every dynamical system is conjugate to the system defined by inverse sequence of twinned graph homomorphisms. 
\end{thm}
\begin{rem}
	Due to non constructive nature of relation given by Definition \ref{df1}, it is hard to explicitly construct examples of even simple dynamical system encoded as inverse limit of twinned graph homomorphisms. It seems that if one was to apply Theorem \ref{mainres} it would be in a non constructive setting, for example like in paper \cite{Shimomura2014}.
\end{rem}
\textbf{Acknowledgements} The last part of proof of Theorem \ref{space->graph,nonzero} was author's undergraduate diploma project, supervised by Dominik Kwietniak, I am grateful for his support. The rest of this note was written during authors first years of PhD studies at AGH Doctoral School. I would like to thank my advisor Jan P. Boro\'{n}ski, who suggested publishing above results and provided with many helpful comments and modifications. This work was partially supported by the National Science Centre, Poland (NCN), ~grant no. 2019/34/E/ST1/00237.
\bibliography{main}

\begin{thebibliography}{BKO19b}

\bibitem[BD12]{BD}
Nilson~C. Bernardes, Jr. and Udayan~B. Darji.
\newblock Graph theoretic structure of maps of the {C}antor space.
\newblock {\em Adv. Math.}, 231(3-4):1655--1680, 2012.

\bibitem[BKO18]{Boroski2017EdreisCR}
Jan~P. Boro\'{n}ski, Ji\v{r}\'{\i} Kupka, and Piotr Oprocha.
\newblock Edrei's conjecture revisited.
\newblock {\em Ann. Henri Poincar\'{e}}, 19(1):267--281, 2018.

\bibitem[BKO19a]{cantorslow}
Jan~P. {Boro\'nski}, Ji\v{r}\'{\i} {Kupka}, and Piotr {Oprocha}.
\newblock {All minimal Cantor systems are slow}.
\newblock {\em {Bull. Lond. Math. Soc.}}, 51(6):937--944, 2019.

\bibitem[BKO19b]{boronski_kupka_oprocha_2019}
Jan~P. Boro\'{n}ski, Ji\v{r}\'{\i} Kupka, and Piotr Oprocha.
\newblock A mixing completely scrambled system exists.
\newblock {\em Ergodic Theory Dynam. Systems}, 39(1):62--73, 2019.

\bibitem[CHL14]{CHL}
Alex Clark, Steven Hurder, and Olga Lukina.
\newblock Shape of matchbox manifolds.
\newblock {\em Indag. Math. (N.S.)}, 25(4):669--712, 2014.

\bibitem[GM06]{Gambaudo2006AlgebraicTF}
Jean-Marc Gambaudo and Marco Martens.
\newblock Algebraic topology for minimal {C}antor sets.
\newblock {\em Ann. Henri Poincar\'{e}}, 7(3):423--446, 2006.

\bibitem[GM20]{goodmedd}
Chris Good and Jonathan Meddaugh.
\newblock Shifts of finite type as fundamental objects in the theory of
  shadowing.
\newblock {\em Invent. Math.}, 220(3):715--736, 2020.

\bibitem[Shi14]{Shimomura2014}
Takashi Shimomura.
\newblock Special homeomorphisms and approximation for {C}antor systems.
\newblock {\em Topology Appl.}, 161:178--195, 2014.

\bibitem[Shi16a]{zerodim-takashi-BV}
Takashi Shimomura.
\newblock {A Bratteli--Vershik representation for all zero-dimensional
  systems}.
\newblock {\em arXiv}, 03 2016.

\bibitem[Shi16b]{Shimomura_2015}
Takashi Shimomura.
\newblock The construction of a completely scrambled system by graph covers.
\newblock {\em Proc. Amer. Math. Soc.}, 144(5):2109--2120, 2016.

\bibitem[Shi16c]{Shimomura2016ergodicity}
Takashi Shimomura.
\newblock Graph covers and ergodicity for zero-dimensional systems.
\newblock {\em Ergodic Theory Dynam. Systems}, 36(2):608--631, 2016.

\bibitem[Shi16d]{SHIMOMURA201663}
Takashi Shimomura.
\newblock Zero-dimensional almost 1--1 extensions of odometers from graph
  coverings.
\newblock {\em Topology Appl.}, 209:63--90, 2016.

\bibitem[{Shi}20]{Shimomura_2020}
Takashi {Shimomura}.
\newblock {Bratteli-Vershik models and graph covering models}.
\newblock {\em {Adv. Math.}}, 367:54, 2020.
\newblock Id/No 107127.

\end{thebibliography}
\bibliographystyle{alpha}
\end{document}